\documentclass[leqno,11pt]{amsart}
\usepackage{amssymb,amsthm,wasysym}

\theoremstyle{plain}

\allowdisplaybreaks

\newtheorem{thm}{Theorem}[section]
\newtheorem{cor}[thm]{Corollary}

\newtheorem{lem}[thm]{Lemma}

\newtheorem*{ex*}{Example}

\numberwithin{equation}{section}

\newcommand{\N}{\mathbb{N}}
\newcommand{\R}{\mathbb{R}_{+}^{d}}
\newcommand{\RR}{\mathbb{R}^{d}}

\def\a{\alpha}
\def\z{\zeta}

\def\t{\theta}
\def\vt{\vartheta}
\def\vr{\varrho}
\def\eps{\varepsilon}
\def\b{(\z,q_{\pm})}

\def\eee{(\z,q_{\pm}({x \vee x'},y,s))}

\def\1{\Pi_{\alpha+\mathbf{1}+\eps}(ds)}

\DeclareMathOperator{\e}{{\mathbb{E}\textrm{{xp}}}}
\DeclareMathOperator{\ee}{{\mathbb{E}\textrm{\emph{xp}}}}
\DeclareMathOperator{\lo}{{\mathbb{L}\textrm{og}}}
\DeclareMathOperator{\loo}{{\mathbb{L}\textrm{\emph{og}}}}

\DeclareMathOperator{\support}{supp}

\begin{document}

%42C05 Orthogonal functions and polynomials, general theory
%42C20 Other transformations of harmonic type
%42B15  	Multipliers
%42B20  	Singular and oscillatory integrals (Calderón-Zygmund, etc.)
%42B25  	Maximal functions, Littlewood-Paley theory

\thanks{
The first-named author was partially supported by MNiSW Grant N N201 417839.
}

\title[Calder\'on-Zygmund operators related to Laguerre expansions]
{Calder\'on-Zygmund operators related to Laguerre function expansions of convolution type}

\author[A. Nowak]{Adam Nowak}
\address{Adam Nowak, \newline
			Instytut Matematyczny,
      Polska Akademia Nauk, \newline
      \'Sniadeckich 8,
      00--956 Warszawa, Poland \newline
			\indent and \newline
			Instytut Matematyki i Informatyki,
      Politechnika Wroc\l{}awska,       \newline
      Wyb{.} Wyspia\'nskiego 27,
      50--370 Wroc\l{}aw, Poland      
      }
\email{adam.nowak@pwr.wroc.pl}

\author[T. Szarek]{Tomasz Szarek}
\address{Tomasz Szarek,     \newline
      ul. W. Rutkiewicz 29\slash 43,
      PL-50--571 Wroc\l{}aw, Poland
      }
\email{szarektomaszz@gmail.com}

\begin{abstract}
We develop a technique of proving standard estimates in the setting of Laguerre function expansions
of convolution type, which works for all admissible type multi-indices $\alpha$ in this context.
This generalizes a simpler method existing in the literature, but being valid for a restricted range
of $\alpha$. As an application, we prove that several fundamental operators in harmonic analysis
of the Laguerre expansions, including maximal operators related to the heat and Poisson semigroups,
Riesz transforms, Littlewood-Paley-Stein type square functions and multipliers of Laplace and
Laplace-Stieltjes transforms type, are (vector-valued) Calder\'on-Zygmund operators in the sense 
of the associated space of homogeneous type.
\end{abstract}

\maketitle

%%%%%%%%%%%%%%%%%%%%%%%%%%%%%%%%%%%%%%%%%%%%%%%%%%%%%%%%%%%%%%%%%%%%%%%%%%%%%%%%%%%%
\section{Introduction and preliminaries} \label{sec:intro}
%%%%%%%%%%%%%%%%%%%%%%%%%%%%%%%%%%%%%%%%%%%%%%%%%%%%%%%%%%%%%%%%%%%%%%%%%%%%%%%%%%%%%

Let $d \ge 1$ and $\alpha=(\alpha_1,\ldots,\alpha_d) \in (-1,\infty)^d$.
We shall work on the space $\R=(0,\infty)^d$ equipped with the measure 
$$
\mu_{\alpha}(dx) = x_1^{2\alpha_1+1}\cdot \ldots \cdot x_d^{2\alpha_d+1}\, dx
$$
and with the Euclidean norm $|\cdot|$. Since $\mu_{\alpha}$ satisfies the doubling condition,
the triple $(\R,d\mu_{\alpha},|\cdot|)$ forms the space of homogeneous type
in the sense of Coifman and Weiss \cite{CW}. The Laguerre operator
$$
L_{\alpha} = -\Delta + |x|^2 - \sum_{i=1}^d \frac{2\alpha_i+1}{x_i} \, \frac{\partial}{\partial x_i}
$$
is symmetric and positive in $L^2(d\mu_{\alpha})$, and it has a natural self-adjoint extension 
$\mathcal{L}_{\alpha}$ whose spectral decomposition is discrete and is given by the Laguerre functions
of convolution type $\ell_k^{\alpha}$, see \cite{NS1}. The associated heat semigroup
$\{\exp(-t\mathcal{L}_{\alpha})\}$ has an integral representation, and the Laguerre heat kernel is
known explicitly, see \cite[Section 2]{NS1}, to be
\begin{equation*}
G^\alpha_t(x,y)=(\sinh 2t)^{-d}\exp\Big({-\frac{1}{2} \coth(2t)\big(|x|^{2}+|y|^{2}\big)}\Big)
\prod^{d}_{i=1} (x_i y_i)^{-\alpha_i} I_{\alpha_i}\left(\frac{x_i y_i}{\sinh 2t}\right),
\end{equation*}
with $I_{\nu}$ denoting the modified Bessel function of the first kind and order $\nu$; as a function
on $\mathbb{R}_+$, $I_{\nu}$ is real, positive and smooth for any $\nu > -1$, cf. \cite{Wat}.

The main objective of this paper is to develop, for \emph{arbitrary} $\alpha \in (-1,\infty)^d$,
a technique of proving standard estimates, see \eqref{gr}-\eqref{grad} below,
for various kernels expressible via $G_t^{\alpha}(x,y)$.
Typical and important examples here are kernels associated with the Laguerre 
heat and Poisson maximal operators,
Riesz-Laguerre transforms, Littlewood-Paley-Stein type square functions and multipliers of Laplace
and Laplace-Stieltjes transforms type. The multiplier operators just mentioned 
cover as special cases imaginary powers of $\mathcal{L}_{\alpha}$ and the related fractional integrals.

For the restricted range of $\alpha \in [-1\slash 2,\infty)^d$, the problem was treated by
Nowak and Stempak \cite{NS1}. The idea standing behind the method presented in \cite{NS1} has roots
in Sasso's paper \cite{Sa} and it is based on Schl\"afli's Poisson type representation for 
the Bessel function (see \cite[Chapter VI, Section 6$\cdot$15]{Wat} and \cite[Section 5]{NS1})
\begin{equation} \label{Schlafli}
I_{\nu}(z) = z^{\nu} \int_{-1}^1 \exp({-z s})\, \Pi_{\nu}(ds), \qquad
    |\arg z| < \pi, \quad \nu \ge -\frac{1}{2},
\end{equation}
where the measure $\Pi_{\nu}$ is given by the density
$$
\Pi_{\nu}(ds) = \frac{(1-s^2)^{\nu-1\slash 2}ds}{\sqrt{\pi} 2^{\nu}\Gamma{(\nu+1\slash 2)}},
    \qquad \nu > -1\slash 2,
$$
and in the limit case $\Pi_{-1\slash 2}$ becomes the atomic measure defined as
the sum of unit point masses at $-1$ and $1$ divided by $\sqrt{2\pi}$.
Assuming that $\alpha \in [-1\slash 2,\infty)^d$,
Schl\"afli's formula allows to write the heat kernel in the following symmetric way:
\begin{equation} \label{Gzeta}
G^{\alpha}_t(x,y) = \Big( \frac{1-\zeta^2}{2\zeta}\Big)^{d + |\alpha|} \int_{[-1,1]^d}
    \exp\Big(-\frac{1}{4\zeta} q_{+}(x,y,s) - \frac{\zeta}{4} q_{-}(x,y,s)\Big) \, \Pi_{\alpha}(ds),
\end{equation}
where $|\a|=\a_1 + \ldots + \a_d$, $\Pi_{\alpha}$ stands for the product measure 
$\bigotimes_{i=1}^d \Pi_{\alpha_i}$,
$$
q_{\pm}(x,y,s) = |x|^2 + |y|^2 \pm 2 \sum_{i=1}^d x_i y_i s_i, \qquad x,y \in \R, \quad s \in [-1,1]^d,
$$
and $t$ is related to $\zeta$ by $\zeta = \tanh t$; equivalently,
\begin{equation} \label{tzeta}
t = t(\zeta) = \frac{1}{2} \log \frac{1+ \zeta}{1-\zeta}, \qquad \zeta \in (0,1).
\end{equation}
This representation of the heat kernel turned out to be particularly well suited for considerations
connected with applications of the Calder\'on-Zygmund theory. The essence and convenience of
the technique derived in \cite{NS1} lies in the fact that the integral against $\Pi_{\alpha}(ds)$
occurring in kernels defined via $G_t^{\alpha}(x,y)$ can be handled independently of the integrand. 
Then expressions
one has to estimate are relatively simple and contain no transcendental functions.
Unfortunately, the restriction $\alpha \in [-1\slash 2,\infty)^d$ resulting from Schl\"afli's
formula cannot be released in a straightforward manner.

To solve the problem and cover in a unified way 
all $\alpha \in (-1,\infty)^d$ we combine \eqref{Schlafli} with the recurrence relation
(cf. \cite[Chapter III, Section 3$\cdot$71]{Wat})
\begin{equation} \label{Irec}
I_{\nu}(z) = \frac{2(\nu+1)}{z} I_{\nu+1}(z) + I_{\nu+2}(z),
\end{equation}
as suggested vaguely in \cite[p.\,666]{NS1}. This leads to a representation of $G_t^{\alpha}(x,y)$
as a sum of $2^d$ components, all of them being similar to the expression in \eqref{Gzeta},
see Section \ref{sec:prep}. Then each component is analyzed by means of a suitable generalization 
of the strategy employed in \cite{NS1}. However, 
the technical side of the present paper is considerably more involved than that of
\cite{NS1} and also some essentially new arguments are required.

As an application of the presented technique, we prove that the maximal operators of the heat and
Poisson semigroups, Riesz-Laguerre transforms, Littlewood-Paley-Stein type square functions and
multipliers of Laplace and Laplace-Stieltjes transforms type are, or can be viewed as, 
Calder\'on-Zygmund operators
in the sense of the space $(\R,d\mu_{\alpha},|\cdot|)$, see Theorem \ref{thm:main}. 
This recovers and extends
to all $\alpha \in (-1,\infty)^d$ known results for $\alpha \in [-1\slash 2,\infty)^d$ obtained in
\cite{NS1} for the maximal operators and Riesz transforms, in \cite{Sz} for vertical and horizontal
$g$-functions of order one, and in \cite{Sz2} for Laplace type multipliers of both types.
Moreover, here we also deal with $g$-functions of arbitrary orders and mixed vertical and horizontal
components, which were not investigated earlier in the Laguerre context.
Noteworthy, our technique is well suited to a wider variety of operators, including more general
forms of the $g$-functions and Lusin's area type integrals.

It is remarkable that recently, in a similar spirit, analogous techniques have been developed
in the Bessel setting (the context of the Hankel transform)
by Betancor, Castro and Nowak \cite{BCN}, and in the more complex Jacobi
setting by Nowak and Sj\"ogren \cite{NoSj}. However, in \cite{BCN}, as well as in \cite{NoSj},
ranges of admissible type indices are restricted, as it took place in the Laguerre situation in
\cite{NS1,Sz,Sz2}. The results of this paper show how to remove the restriction in the Bessel setting,
and give an intuition how it could be done in the Jacobi situation. Furthermore, they also
suggest that various results obtained recently in the context of the Dunkl harmonic oscillator
and the associated group of reflections isomorphic to $\mathbb{Z}_2^d$, see \cite{NS3,NS2,Sz1,Sz2},
hold for more general (not necessarily positive) multiplicity functions.

The paper is organized as follows.
In Section \ref{sec:prep} we gather various facts and preparatory results needed for kernel estimates.
In Section \ref{sec:ker} we demonstrate our technique by proving standard estimates for kernels
associated with the operators mentioned above. Finally, in Section \ref{sec:CZ} we conclude
that the operators in question can be interpreted as Calder\'on-Zygmund operators.

\textbf{Notation.}
Throughout the paper we use a fairly standard notation with essentially all symbols referring to the space
of homogeneous type $(\R,d\mu_{\alpha},|\cdot|)$. For the sake of clarity, we now explain all
symbols and relations that might lead to a confusion. Given $x,y \in \R$, $\beta \in \RR$ and a multi-index 
$n \in \mathbb{N}^d$, $\mathbb{N}=\{0,1,2,\ldots\}$, we denote
\begin{align*}
e_j & \equiv \textrm{$j$th coordinate vector in $\R$},  \\
\mathbf{1} & = (1,\ldots,1) \in \mathbb{N}^d,\\
|n| & = n_1 + \ldots + n_d, \qquad \textrm{(length of $n$)}\\
B(x,r) & = \{y \in \R : |x-y|<r\}, \qquad r>0, \qquad 
	\textrm{(balls in $\R$)}\\
xy & = (x_1 y_1,\ldots, x_d y_d),\\
x^\beta & = x_1^{\beta_1}\cdot \ldots \cdot x_d^{\beta_d},\\
x \le y & \equiv x_i \le y_i, \qquad i=1,\ldots,d,\\
x \vee y & = (\max\{x_1,y_1\},\ldots,\max\{x_d,y_d\}),\\
\partial_{x_i} & = \partial \slash \partial x_i, \qquad i=1,\ldots,d,\qquad 
	\textrm{(ordinary partial derivatives)}\\
\partial_x^n & = \partial_{x_1}^{n_1} \circ \ldots \circ \partial_{x_d}^{n_d},\\
\delta_{x_i} & = \partial_{x_i} + x_i, \qquad i=1,\ldots,d, \qquad \textrm{(Laguerre partial derivatives)}\\
\delta_x^n & = \delta_{x_1}^{n_1}\circ \ldots \circ \delta_{x_d}^{n_d},\\
(\partial^k_x F)^n & = (\partial^k_{x_1}F)^{n_1} \cdot \ldots \cdot (\partial^k_{x_d}F)^{n_d},
\qquad k=1,2,\ldots,
\end{align*}
where in the last identity $F$ is a suitable function on $\R$ defined in a moment.
 
Further, we also introduce the following notation and abbreviations:
\begin{align*}
q_{\pm} & = q_{\pm}(x,y,s),  \\
\e\b & = \exp\Big(-\frac{1}{4\zeta} q_{+} - \frac{\zeta}{4} q_{-}\Big), \\ 
F & = F(\zeta,q_{\pm}) = \ln \e\b, \\
\lo(\zeta) & = \log\frac{1+\zeta}{1-\zeta},\\
\Psi_{\pm}^j & = \Psi_{\pm}^j(x,y,s) = x_j\pm y_j s_j, \qquad j=1,\ldots,d,\\
\Psi_{\pm} & = (\Psi_{\pm}^1,\ldots,\Psi_{\pm}^d),\\
\Phi_{\pm}^j & = \Phi_{\pm}^j(x,y,s) = y_j\pm x_j s_j, \qquad j=1,\ldots,d,
\end{align*}
where $x,y \in \R$, $s \in [-1,1]^d$ and $\zeta \in (0,1)$.

While writing estimates, we will use the notation $X \lesssim Y$ to
indicate that $X \le CY$ with a positive constant $C$ independent of significant quantities. We shall
write $X \simeq Y$ when simultaneously $X \lesssim Y$ and $Y \lesssim X$.

%%%%%%%%%%%%%%%%%%%%%%%%%%%%%%%%%%%%%%%%%%%%%%%%%%%%%%%%%%%%%%%%%%%%%%%%%%%%%%%%%%%%
\section{Preparatory facts and results} \label{sec:prep}
%%%%%%%%%%%%%%%%%%%%%%%%%%%%%%%%%%%%%%%%%%%%%%%%%%%%%%%%%%%%%%%%%%%%%%%%%%%%%%%%%%%%%

Let $\alpha \in (-1,\infty)^d$. 
By means of \eqref{Irec} and \eqref{Schlafli} the Laguerre heat kernel can be written as
\begin{equation} \label{G}
G_t^{\alpha}(x,y) = \sum_{\eps \in \{0,1\}^d} C_{\alpha,\eps}
	\Big( \frac{1-\zeta^2}{2\zeta}\Big)^{d+|\alpha|+2|\eps|} (xy)^{2\eps} \int
	\exp\Big(-\frac{1}{4\zeta} q_{+} - \frac{\zeta}{4} q_{-}\Big) \, \1,
\end{equation}
where $C_{\alpha,\eps} = [ 2(\alpha+\mathbf{1})]^{\mathbf{1}-\eps}$
and $t$ and $\zeta$ are related as in \eqref{tzeta}. Here and later on, for the sake of brevity, 
we omit the set of integration $[-1,1]^d$ in integrals against $\1$.

The following generalization of \cite[Proposition 5.9]{NS1} is a crucial point in our method of estimating
kernels. It establishes a relation between expressions involving certain integrals with respect to
$\Pi_{\alpha+\mathbf{1}+\eps}(ds)$ and the standard estimates for the space $(\R,d\mu_{\alpha},|\cdot|)$.
\begin{lem} \label{lem:bridge}
Let $\alpha \in (-1,\infty)^d$. Assume that $\xi,\kappa \in [0,\infty)^d$ are fixed and such that
$\alpha + \xi + \kappa \in [-1\slash 2,\infty)^d$. Then, uniformly in $x,y \in \R$, $x\neq y$,
\begin{align*}
(x+y)^{2\xi} \int \Big( \frac{1}{q_{+}}\Big)^{d + |\alpha|+ |\xi|} \Pi_{\alpha+\xi+\kappa}(ds)
& \lesssim \frac{1}{\mu_{\alpha}(B(x,|x-y|))}, \\ 
(x+y)^{2\xi} \int \Big( \frac{1}{q_{+}}\Big)^{d + |\alpha|+ |\xi|+1\slash 2} \Pi_{\alpha+\xi+\kappa}(ds)
& \lesssim \frac{1}{|x-y|\,\mu_{\alpha}(B(x,|x-y|))}. 
\end{align*}
\end{lem}
To prove this we need two auxiliary results. 
The first one is a natural extension of \cite[Proposition 3.2]{NS1}.
\begin{lem} \label{ball}
Let $\alpha \in (-1,\infty)^d$. Then
$$
\mu_{\alpha}(B(x,r)) \simeq r^d \prod_{i=1}^d (x_i + r)^{2\alpha_i+1}, \qquad x \in \R, \quad r>0.
$$
\end{lem}

\begin{proof}
Let $x \in \R$ and $r>0$. Given $\eps \in \{0,1\}^d$, we consider the cube
$Q_{\eps}(x,r)$ being a product of the intervals  $[x_i+\eps_i r, x_i + r + \eps_i r]$, $i=1,\ldots,d$.
Since $\mu_{\alpha}$ possesses the doubling property, for each $\eps \in \{0,1\}^d$ we have
$$
\mu_{\alpha}(Q_{\eps}(x,r)) \simeq \mu_{\alpha}(B(x,r)), \qquad x \in \R, \quad r>0.
$$
Now for a fixed $\alpha$ we choose $\eps$ such that $\eps_i=1$ when $\alpha_i < -1\slash 2$ and $\eps_i=0$
if $\alpha_i \ge -1\slash 2$. By the mean value theorem for integration,
$$
\mu_{\alpha}(Q_{\eps}(x,r)) \simeq r^d \theta^{2\alpha+\mathbf{1}}, \qquad x \in \R, \quad r>0,
$$
where $\theta=\theta(x,r)$ is a point in $Q_{\eps}(x,r)$. But the right-hand side here is, by the choice
of $\eps$, dominated by $r^d \prod_{i=1}^d(x_i+r)^{2\alpha_i+1}$. A similar argument shows that
$$
\mu_{\alpha}(Q_{\mathbf{1}-\eps}(x,r)) \gtrsim r^d \prod_{i=1}^d (x_i+r)^{2\alpha_i+1}, \qquad x \in \R,
	\quad r > 0.
$$ 
The conclusion follows.
\end{proof}

The second result we need is a slightly more general version of \cite[Lemma 5.8]{NS1}.
\begin{lem} \label{58gen}
Let $a\ge -1\slash 2$, $b \ge 0$ and $\lambda > 0$ be fixed. Then
$$
\int_{-1}^1 \frac{\Pi_{a+b}(ds)}{(A-Bs)^{a+1\slash 2+\lambda}} 
	\lesssim \frac{1}{A^{a+1\slash 2}(A-B)^{\lambda}}, \qquad A>B>0.
$$
\end{lem}

\begin{proof}
When $b=0$ this is precisely \cite[Lemma 5.8]{NS1}. Using this special case we can write
$$
\int_{-1}^1 \frac{\Pi_{a+b}(ds)}{(A-Bs)^{a+1\slash 2+\lambda}} \le (A+B)^b
	\int_{-1}^1 \frac{ \Pi_{a+b}(ds)}{(A-Bs)^{a+b+1\slash 2+\lambda}}
	\lesssim \frac{(A+B)^b}{A^{a+b+1\slash 2}(A-B)^{\lambda}}.
$$
Since $(A+B) \simeq A$, the desired bound follows.
\end{proof}

\begin{proof}[Proof of Lemma \ref{lem:bridge}]
It suffices to verify the first estimate of the lemma. Then the second one follows immediately by
observing that $q_{+} \ge |x-y|^2$. Further, our task can be reduced to showing that
\begin{equation} \label{red}
\int \Big( \frac{1}{q_{+}} \Big)^{d+|\alpha|} \, \Pi_{\alpha+\kappa}(ds) \lesssim
	\frac{1}{\mu_{\alpha}(B(x,|x-y|))}, \qquad x,y \in \R, \quad x \neq y,
\end{equation}
provided that $\alpha + \kappa \in [-1\slash 2,\infty)^d$. Indeed, replacing in \eqref{red}
$\alpha$ by $\alpha+\xi$ and using Lemma \ref{ball} we get
\begin{align*}
(x+y)^{2\xi} \int \Big(\frac{1}{q_+} \Big)^{d+|\alpha+\xi|} \, \Pi_{\alpha+\xi+\kappa}(ds) & \lesssim
	(x+y)^{2\xi} \frac{1}{\mu_{\alpha+\xi}(B(x,|x-y|))} \\
& \simeq \frac{(x+y)^{2\xi}}{|x-y|^d \prod_{i=1}^d (x_i+|x-y|)^{2(\alpha_i+\xi_i)+1}} \\
& \lesssim \frac{1}{|x-y|^d \, \prod_{i=1}^d (x_i+|x-y|)^{2\alpha_i+1}} 
\simeq \frac{1}{\mu_{\alpha}(B(x,|x-y|))},
\end{align*}
where the third relation follows from the bound $x_i+y_i \lesssim x_i + |x-y|$.

It remains to verify \eqref{red}. Let $\mathcal{I}_{\alpha} = \{j: \alpha_j < -1\slash 2\}$. 
Taking into account Lemma \ref{ball}, the symmetry of $\Pi_{\alpha+\kappa}$ and the estimate
$$
\frac{1}{|x-y|^{2\alpha_j+1}} \le \frac{1}{(x_j+|x-y|)^{2\alpha_j+1}}, \qquad \alpha_j < -1\slash 2,
$$
we see that it is enough to show the bound
\begin{align}\nonumber 
& \int \Big(\frac{1}{q_{-}} \Big)^{d+|\alpha|}\, \Pi_{\alpha+\kappa}(ds)\\ & \quad \lesssim 
	\frac{1}{|x-y|^d \prod_{i \in \mathcal{I}_{\alpha}}|x-y|^{2\alpha_i+1} 
		\prod_{j \notin \mathcal{I}_{\alpha}} (x_j+|x-y|)^{2\alpha_j+1}}, 
		\qquad x,y \in \R, \quad x \neq y, \label{red2}
\end{align}
with the usual convention concerning empty products. Here, without any loss of generality,
we may assume that $\mathcal{I}_{\alpha}=\{1,\ldots, k\}$ for some $k=0,1,\ldots,d$ 
(by convention, $k=0$ corresponds to
$\mathcal{I}_{\alpha}=\emptyset$). Then proving \eqref{red2} consists of two steps.  

\noindent \textbf{Step 1.} If $\mathcal{I}_{\alpha}=\{1,\ldots,d\}$, we go immediately to Step 2. 
Otherwise we proceed as in the proof of \cite[Proposition 5.9]{NS1}, using Lemma \ref{58gen} instead
of \cite[Lemma 5.8]{NS1}. This either produces directly \eqref{red2} in case
$\mathcal{I}_{\alpha}=\emptyset$, or leads to the estimate
\begin{align*}
& \int \Big(\frac{1}{q_{-}} \Big)^{d+|\alpha|}\, \Pi_{\alpha+\kappa}(ds) \lesssim 
	\frac{1}{\prod_{j=k+1}^d (x_j+|x-y|)^{2\alpha_j+1}} \\ & \quad \times \int_{[-1,1]^k}
		\frac{1}{(|x|^2+|y|^2-2\sum_{i=1}^k x_i y_i s_i - 2\sum_{j=k+1}^d x_j y_j)^{d+\sum_{i=1}^k \alpha_i
		-(d-k)\slash 2}}\, \Pi_{\tilde{\alpha}+\tilde{\kappa}}(d\tilde{s}),
\end{align*}
where $\tilde{\cdot}$ indicates the restriction to the first $k$ axes.

\noindent \textbf{Step 2.} Taking into account the last estimate, the fact that the measure
$\Pi_{\tilde{\alpha}+\tilde{\kappa}}$ is finite and the bounds
$$
d + \sum_{i=1}^k \alpha_i - (d-k)\slash 2 \ge (d-k)\slash 2 \ge 0,
$$
$$
|x|^2 + |y|^2 -2 \sum_{i\in \mathcal{I}_{\alpha}} x_i y_i s_i 
	- 2\sum_{j \notin \mathcal{I}_{\alpha}} x_j y_j \ge |x-y|^2,
$$
we conclude that
$$
\int \Big(\frac{1}{q_{-}} \Big)^{d+|\alpha|}\, \Pi_{\alpha+\kappa}(ds) \lesssim
	\frac{1}{\prod_{j \notin \mathcal{I}_{\alpha}}(x_j+|x-y|)^{2\alpha_j+1}}\; 
		\frac{1}{|x-y|^{d+k+\sum_{i=1}^k 2\alpha_i}}.
$$
This implies \eqref{red2}. The proof is finished.
\end{proof}

The remaining part of this section contains lemmas that are needed to control the relevant kernels
and their gradients by means of the estimates from Lemma \ref{lem:bridge}. To prove some of the
technical results below
we will use Fa\`a di Bruno's formula for the $N$th derivative, $N \ge 1$, of the composition
of two functions (see \cite{Jo} for the related references and interesting historical remarks),
\begin{equation} \label{Faa}
\partial_{x}^N(g\circ f)(x) = \sum \frac{N!}{p_1! \cdot \ldots \cdot p_N!} \;\partial^{p_1+\ldots+p_N}
	g\circ f(x) \bigg( \frac{\partial_x^1 f(x)}{1!}\bigg)^{p_1}\cdot \ldots \cdot
	\bigg( \frac{\partial_x^N f(x)}{N!}\bigg)^{p_N},
\end{equation}
where the summation runs over all $p_1,\ldots,p_N \ge 0$ such that $p_1+2p_2+\ldots+N p_N = N$.

\begin{lem} \label{lem:deltax}
Let $d \ge 1$, $n \in \N^d$, $\eps \in \{0,1\}^d$. Then
\begin{align*}
& \delta_x^n \big[ (xy)^{2\eps} \ee\b\big] \\
& \quad = y^{2\eps} \sum_{\eta \in \{0,1,2\}^d} x^{2\eps-\eta\eps}
	\sum_{\substack{k,l \in \N^d \\ k+2l \le n-\eta \eps}} 
	\chi_{\{n \ge \eta\eps\}} P_{n,\eps,\eta,k,l}(x)
	(\partial_x F)^k (\partial_x^2 F)^l \ee\b,
\end{align*}
where 
$$
P_{n,\eps,\eta,k,l}(x) = \prod_{i=1}^d P_{n_i,\eps_i,\eta_i,k_i,l_i}(x_i)
$$
is a product of one-dimensional polynomials of degrees $n_i-\eta_i \eps_i-k_i-2l_i$, respectively.
\end{lem}

\begin{proof}
By the product structure of the expression $(xy)^{2\eps} \e\b$ it is enough to prove the result
in the one-dimensional case. Thus we assume that $d=1$.

Proceeding inductively it is easy to see that
$$
\delta_x^n f = \sum_{m=0}^n P_{n,m}(x) \partial_x^m f,
$$
where $P_{n,m}$ is a polynomial of degree $n-m$. Further, we observe that by Leibniz' rule
$$
\partial_x^m [x^2 f] = x^2\partial_x^m f + 2\,\chi_{\{m \ge 1\}} m\, x \,\partial_x^{m-1} f
	+ \chi_{\{m \ge 2\}} m(m-1)\, \partial_x^{m-2}f.
$$
Finally, taking into account that $\partial_x^3 F = \partial_x^4 F = \ldots = 0$, we deduce from
\eqref{Faa} that
$$
\partial_x^m \e\b = \partial_x^m \exp(F) = \exp(F) \sum_{\substack{k,l \ge 0 \\ k+2l =m}}
	c_{m,k} (\partial_x F)^k (\partial_x^2 F)^l,
$$
where $c_{m,k} \in \mathbb{R}$ are constants. 

These facts altogether imply that for $\eps = 0$,
$$
\delta_x^n \big[ x^{2\eps} \e\b\big] = \sum_{\substack{k,l \ge 0 \\ k+2l \le n}}
	P_{n,k,l}(x) (\partial_x F)^k (\partial_x^2 F)^l \e\b,
$$
and when $\eps =1$,
\begin{align*}
\delta_x^n \big[ x^{2\eps} \e\b\big] & = \sum_{m=0}^n P_{n,m}(x)
	\sum_{\eta =0,1,2} C_{m,\eta} \chi_{\{m \ge \eta\}}
	x^{2-\eta} \partial_x^{m-\eta} \e\b \\
& = \sum_{\eta = 0,1,2} x^{2-\eta} \sum_{\substack{k,l \ge 0 \\ k+2l \le n - \eta}} \chi_{\{n \ge \eta\}}
	P_{n,\eta,k,l}(x) (\partial_x F)^k (\partial_x^2 F)^l \e\b,
\end{align*}
where $P_{n,k,l}$ and $P_{n,\eta,k,l}$ are polynomials of degrees $n-k-2l$ and $n-\eta-k-2l$, respectively.
Combining together the formulas for $\eps = 0$ and $\eps = 1$ produces
$$
\delta_x^n \big[ x^{2\eps} \e\b\big] 
= \sum_{\eta = 0,1,2} x^{2\eps-\eta\eps}
	\sum_{\substack{k,l \ge 0 \\ k+2l \le n-\eta \eps}} 
	\chi_{\{n \ge \eta\eps\}} P_{n,\eps,\eta,k,l}(x)
	(\partial_x F)^k (\partial_x^2 F)^l \e\b
$$
with $P_{n,\eps,\eta,k,l}$ being a polynomial of degree $n-\eta\eps-k-2l$.
The conclusion follows.
\end{proof}

\begin{lem} \label{lem:dtdx}
Let $d \ge 1$, $\alpha\in (-1,\infty)^d$, $m \in \N\setminus\{ 0 \}$, $n \in \N^d$, $\eps \in \{0,1\}^d$.
Then
\begin{align*}
& \partial^m_t \delta_x^n \bigg[\Big(\frac{1-\zeta^{2}}{\zeta}\Big)^{d+|\alpha|+2|\eps|}
	(xy)^{2\eps} \ee\b\bigg] \\
& \quad =
y^{2\eps} \sum_{\substack{w\in \N^{m} \\ w_1+\ldots+ m w_m=m}} Q_{w}(\z) 
\sum_{\eta \in \{0,1,2\}^d} x^{2\eps-\eta\eps}
	\sum_{\substack{k,l \in \N^d \\ k+2l \le n-\eta \eps}} 
	\chi_{\{n \ge \eta\eps\}} P_{n,\eps,\eta,k,l}(x)\\
&\qquad\times
	\sum_{i=-|l|}^{|l|}\sum_{\substack{v \in \N^d \\ v \le k}}
	\sum_{\substack{j,p,r \in \N \\ j+p+r \le |w|}}
C_{m,w,j,p,r,d,\a,\eps,i,v,k,l}\,
	\big(1-\z^2\big)^{d+|\a|+2|\eps|+|w|-j}\,
\z^{-d-|\a|-2|\eps|+i-|w|+2j}\\
&\qquad\times
\Big(\frac{q_{+}}{\z}\Big)^{p}(\z q_{-})^{r}
\Big(\frac{1}{\z}\Psi_{+}\Big)^{v} 
(\z \Psi_{-})^{k-v} \ee\b,
\end{align*}
where $\zeta = \zeta(t) = \tanh t$, $Q_w$ are polynomials, $C_{m,w,j,p,r,d,\a,\eps,i,v,k,l}$ 
are constants and $P_{n,\eps,\eta,k,l}$ are the polynomials from Lemma \ref{lem:deltax}.
\end{lem}

\begin{proof}
For the sake of lucidity we denote 
\begin{align*}
& \Upsilon_{\!n}(x,y,\z(t),s)\\
& =
\Big(\frac{1-\zeta^{2}}{\zeta}\Big)^{d+|\alpha|+2|\eps|}
	 \delta_x^n \big[ (xy)^{2\eps} \e\b\big]\\
& =
\Big(\frac{1-\zeta^{2}}{\zeta}\Big)^{d+|\alpha|+2|\eps|}
	y^{2\eps} \sum_{\eta \in \{0,1,2\}^d} x^{2\eps-\eta\eps}
	\sum_{\substack{k,l \in \N^d \\ k+2l \le n-\eta \eps}} 
	\chi_{\{n \ge \eta\eps\}} P_{n,\eps,\eta,k,l}(x)
	(\partial_x F)^k (\partial_x^2 F)^l \e\b,
\end{align*}
where the second identity is a consequence of Lemma \ref{lem:deltax}. Applying Fa\`a di Bruno's 
formula \eqref{Faa} we obtain
$$
\partial_{t}^{m}\Upsilon_{\!n}(x,y,\z(t),s)
=
\sum_{\substack{w\in \N^{m} \\ w_1+\ldots+ m w_m=m}} C_{w} \, \partial_{\z}^{|w|}\Upsilon_{\!n}(x,y,\z,s)
	\Big[ \big( \partial_t^1\z(t)\big)^{w_1}\cdot \ldots \cdot \big( \partial_t^m\z(t)\big)^{w_m} \Big].	
$$

We first analyze the expression in square brackets above. 
By induction it follows that
\begin{align*}
\partial_t^u\z(t)\big|_{t=t(\z)} = (1-\z^2)R_u(\z),\qquad u = 1,2,\ldots,
\end{align*}
where $R_u$ are polynomials. Thus we get
\begin{equation}\label{eq1}
\big( \partial_t^1\z(t)\big)^{w_1}\cdot \ldots \cdot \big( \partial_t^m\z(t)\big)^{w_m}
= (1-\z^2)^{|w|}\,Q_{w}(\z),
\end{equation}
where $Q_w$ are polynomials. 

Next we deal with $\partial_{\z}^{u}\Upsilon_{\!n}(x,y,\z,s)$ for $u \in \N$.
Proceeding inductively one checks that for any $M,W\in \mathbb{R}$
\begin{align}\nonumber
& \partial_{\z}^{u}\big[(1-\z^2)^M \z^W \e\b \big]\\\label{eq2}
& \quad =
	\sum_{\substack{j,p,r \in \N \\ j+p+r \le u}}
	C_{u,j,p,r,M,W} (1-\z^2)^{M-j} \z^{W-u+2j}
	\Big(\frac{q_{+}}{\z}\Big)^{p}(\z q_{-})^{r}
	\e\b,
\end{align}
where $C_{u,j,p,r,M,W}\in \mathbb{R}$ are constants. Furthermore, since
$$
\partial_{x_j}F=-\frac{1}{2\z}\Psi_{+}^{j}-\frac{\z}{2}\Psi_{-}^{j},\qquad
\partial_{x_j}^{2}F=-\frac{1}{2\z}-\frac{\z}{2},\qquad \quad j=1,\ldots,d,
$$
by means of Newton's formula we infer that
\begin{align*}
& \Big(\frac{1-\zeta^{2}}{\zeta}\Big)^{d+|\alpha|+2|\eps|}
	(\partial_x F)^k (\partial_x^2 F)^l\\
&\quad =
(1-\z^2)^{d+|\a|+2|\eps|}\sum_{i=-|l|}^{|l|}
	\sum_{\substack{v \in \N^d \\ v \le k}} C_{i,v,k,l}\, \z^{-d-|\a|-2|\eps|+i}
	\Big(\frac{1}{\z}\Psi_{+}\Big)^{v} 
	(\z \Psi_{-})^{k-v},
\end{align*}
where $C_{i,v,k,l} \in \mathbb{R}$ are constants.
Then using \eqref{eq2} specified to $M=d+|\a|+2|\eps|$, $W=-d-|\a|-2|\eps|+i-|v|+|k-v|$ and $u=|w|$ produces
\begin{align*}
& \partial_{\z}^{|w|}\bigg[\Big(\frac{1-\zeta^{2}}{\zeta}\Big)^{d+|\alpha|+2|\eps|}
	(\partial_x F)^k (\partial_x^2 F)^l \e\b\bigg]\\
& \quad =
\sum_{i=-|l|}^{|l|}
	\sum_{\substack{v \in \N^d \\ v \le k}}
	\sum_{\substack{j,p,r \in \N \\ j+p+r \le |w|}}
C_{w,j,p,r,d,\a,\eps,i,v,k,l}\,(1-\z^2)^{d+|\a|+2|\eps|-j}
\z^{-d-|\a|-2|\eps|+i-|w|+2j}\\
&\qquad\times
\Big(\frac{q_{+}}{\z}\Big)^{p}(\z q_{-})^{r} \Big(\frac{1}{\z}\Psi_{+}\Big)^{v} (\z \Psi_{-})^{k-v} \e\b,
\end{align*}
where $C_{w,j,p,r,d,\a,\eps,i,v,k,l}\in \mathbb{R}$ are constants.

Combining the last identity with \eqref{eq1} leads to the asserted formula.
\end{proof}

\begin{lem} \label{lem:EST}
Let $d \ge 1$, $\alpha \in (-1,\infty)^d$, $n \in \N^d$, $m \in \N$, $\eps \in \{0,1\}^d$. Then
\begin{align}\nonumber
& \bigg| \partial^m_t \delta_x^n \bigg[\Big(\frac{1-\zeta^{2}}{\zeta}\Big)^{d+|\alpha|+2|\eps|}
	(xy)^{2\eps} \ee\b\bigg]\bigg| \\ \label{EST1}
& \quad \lesssim (1-\zeta^2)^{d+|\alpha|+2|\eps|} \; y^{2\eps} \sum_{\eta \in \{0,1,2\}^d}
	x^{2\eps-\eta\eps} \zeta^{-d-|\alpha|-2|\eps|-m-|n|\slash 2+|\eta\eps|\slash 2}
	\sqrt{\ee\b}
\end{align}
and
\begin{align}\nonumber
& \bigg| \nabla_{\! x,y}\,\partial^m_t \delta_x^n
	 \bigg[\Big(\frac{1-\zeta^{2}}{\zeta}\Big)^{d+|\alpha|+2|\eps|}
	(xy)^{2\eps} \ee\b\bigg]\bigg| \\ \label{EST2}
& \quad \lesssim (1-\zeta^2)^{d+|\alpha|+2|\eps|} \Bigg\{ y^{2\eps} \sum_{\eta \in \{0,1,2\}^d}
	x^{2\eps-\eta\eps} \zeta^{-d-|\alpha|-2|\eps|-m-|n|\slash 2+|\eta\eps|\slash 2-1\slash 2}
	\big(\ee\b \big)^{1 \slash 4}\\\nonumber
& \qquad + 
	\sum_{j=1}^{d} \chi_{\{\eps_j=1\}} y^{2\eps-e_j} 
	\sum_{\eta \in \{0,1,2\}^d} x^{2\eps-\eta\eps} 
	\zeta^{-d-|\alpha|-2|\eps|-m-|n|\slash 2+|\eta\eps|\slash 2} 
	\big( \ee\b \big)^{1 \slash 4} \Bigg\},
\end{align}
uniformly in $\zeta \in (0,1)$, $s\in[-1,1]^d$ and $x,y \in \R$; here $\zeta=\zeta(t)=\tanh t$.
\end{lem}

To prove the lemma, we first state some simple auxiliary estimates.
The following is a compilation of \cite[Lemma 4.1, Lemma 4.2]{Sz} and 
\cite[Corollary 5.2, Lemma 5.5 (a)]{NS1}. 
\begin{lem}\label{lem:comp}
Let $b\geq 0$ and $c>0$ be fixed. Then for any $j=1,\ldots,d,$ we have
\begin{align*}
& \emph{(a)} \qquad
	|\Psi_{\pm}^{j}| \leq \sqrt{q_{\pm}} \qquad \textrm{and} \qquad
	|\Phi_{\pm}^{j}| \leq \sqrt{q_{\pm}},\\
& \emph{(b)} \qquad
	\big( A q_{\pm} \big)^{b} \exp \big( -cAq_{\pm} \big) \lesssim 1,\\
& \emph{(c)} \qquad
	\big( |\Psi_{+}^{j}| + |\Phi_{+}^{j}| \big)^{b} \big( \ee\b \big)^{c} 
	\lesssim 
	\z^{b\slash 2},\\
& \emph{(d)} \qquad
	\big( |\Psi_{-}^{j}| + |\Phi_{-}^{j}| \big)^{b} \big( \ee\b \big)^{c} 
	\lesssim 
	\z^{-b\slash 2},\\
& \emph{(e)} \qquad
	(x_j)^b \big( \ee\b \big)^{c} 
		\lesssim 
	\z^{-b\slash 2},
\end{align*}
uniformly in $x,y \in \R$ and $s \in [-1,1]^d$, and also in $A>0$ if (b) is considered,
and in $\z \in (0,1)$ when items (c)-(e) are taken into account.
\end{lem}

\begin{proof}[Proof of Lemma \ref{lem:EST}]
We will verify \eqref{EST1} and \eqref{EST2} for $m>0$.
Analogous arguments combined with Lemma \ref{lem:deltax} rather than Lemma \ref{lem:dtdx}
in the reasoning below justify the case $m=0$.

The proof is based on the explicit formula established in Lemma \ref{lem:dtdx}.
In what follows we use the notation of that lemma without further comments. We first show \eqref{EST1}.
Using Lemma \ref{lem:comp} (b)-(e) and the inequality $\z < 1$, we see that
\begin{align*}
|P_{n,\eps,\eta,k,l} (x)| \big( \e\b \big)^{1 \slash 6} 
& \lesssim \z^{-|n| \slash 2 + |\eta\eps| \slash 2 +|k| \slash 2 +|l|},\\
(1-\z^2)^{|w|-j} & \le 1, \qquad 0 \le j \le |w|\\
\z^{i-|w|+2j} &\le \z^{-|l|-m},\qquad 0 \le j \le |w| \le m, \quad -|l| \le i \le |l|,\\
\Big(\frac{q_{+}}{\z}\Big)^{p} (\z q_{-})^{r} \big( \e\b \big)^{1 \slash 6} 
& \lesssim 1, \qquad 0 \le p,r \le |w|,\\
\Big| \Big(\frac{1}{\z}\Psi_{+}\Big)^{v} 
	(\z \Psi_{-})^{k-v} \Big| \big( \e\b \big)^{1 \slash 6} 
& \lesssim \z^{-|k|\slash 2}, \qquad v \le k \le n,
\end{align*}
where the relations $\lesssim$ hold uniformly in $\z \in (0,1)$, $x,y\in \R$ and $s\in [-1,1]^d$.
Combining Lemma \ref{lem:dtdx} with these estimates and using the fact that the polynomials $Q_{w}$ 
are bounded on $(0,1)$ leads directly to the desired conclusion.

It remains to prove \eqref{EST2}. We have
\begin{align*}
& \bigg| \nabla_{\! x,y}\,\partial^m_t \delta_x^n
	 \bigg[\Big(\frac{1-\zeta^{2}}{\zeta}\Big)^{d+|\alpha|+2|\eps|}
	(xy)^{2\eps} \e\b\bigg]\bigg|\\
& \quad \le 
 \bigg| \nabla_{\! x}\,\partial^m_t \delta_x^n
	 \bigg[\Big(\frac{1-\zeta^{2}}{\zeta}\Big)^{d+|\alpha|+2|\eps|}
	(xy)^{2\eps} \e\b\bigg]\bigg| \\
& \qquad +
 \bigg| \nabla_{\! y}\,\partial^m_t \delta_x^n
	 \bigg[\Big(\frac{1-\zeta^{2}}{\zeta}\Big)^{d+|\alpha|+2|\eps|}
	(xy)^{2\eps} \e\b\bigg]\bigg|
\equiv H_x +H_y.
\end{align*}
We will analyze $H_x$ and $H_y$ separately. 
Treatment of $H_x$ is straightforward. We observe that
$\partial_{x_j}=\delta_{x_j} - x_j$ and since $\partial_{t}^m$ commutes with $\delta_{x_j}$ we get
$$
\partial_{x_j} \partial_{t}^m \delta_{x}^{n} = \partial_{t}^m \delta_{x}^{n+e_j} 
	 - x_j \partial_{t}^m \delta_{x}^{n}, \qquad j=1,\ldots,d.
$$
Thus the required estimate of $H_x$
follows easily from \eqref{EST1} and Lemma \ref{lem:comp} (e) applied with $b=1$ and $c=1 \slash 4$.

To deal with $H_y$, we first differentiate in $y_j$ the formula from Lemma \ref{lem:dtdx}. The result is
\begin{align*}
& \partial_{y_j} \partial^m_t \delta_x^n \bigg[\Big(\frac{1-\zeta^{2}}{\zeta}\Big)^{d+|\alpha|+2|\eps|}
	(xy)^{2\eps} \e\b\bigg] \\
& =
\chi_{\{ \eps_j = 1 \}} 2 y^{2\eps - e_j } \sum_{\substack{w\in \N^{m} \\ 
	w_1+\ldots+ m w_m=m}} Q_{w}(\z) \sum_{\eta \in \{0,1,2\}^d} x^{2\eps-\eta\eps}
	\sum_{\substack{k,l \in \N^d \\ k+2l \le n-\eta \eps}} 
	\chi_{\{n \ge \eta\eps\}} P_{n,\eps,\eta,k,l}(x)\\
&\qquad\times
	\sum_{i=-|l|}^{|l|}\sum_{\substack{v \in \N^d \\ v \le k}}
	\sum_{\substack{j,p,r \in \N \\ j+p+r \le |w|}}
C_{m,w,j,p,r,d,\a,\eps,i,v,k,l}\,
	\big(1-\z^2\big)^{d+|\a|+2|\eps|+|w|-j}\,
	\z^{-d-|\a|-2|\eps|+i-|w|+2j}\\
& \qquad \times
\Big(\frac{q_{+}}{\z}\Big)^{p}(\z q_{-})^{r}
\Big(\frac{1}{\z}\Psi_{+}\Big)^{v} 
(\z \Psi_{-})^{k-v} \e\b \\
& \quad +
y^{2\eps} \sum_{\substack{w\in \N^{m} \\ 
	w_1+\ldots+ m w_m=m}} Q_{w}(\z) \sum_{\eta \in \{0,1,2\}^d} x^{2\eps-\eta\eps}
	\sum_{\substack{k,l \in \N^d \\ k+2l \le n-\eta \eps}} 
	\chi_{\{n \ge \eta\eps\}} P_{n,\eps,\eta,k,l}(x)\\
& \qquad \times
	\sum_{i=-|l|}^{|l|}\sum_{\substack{v \in \N^d \\ v \le k}}
	\sum_{\substack{j,p,r \in \N \\ j+p+r \le |w|}}
C_{m,w,j,p,r,d,\a,\eps,i,v,k,l}\,
	\big(1-\z^2\big)^{d+|\a|+2|\eps|+|w|-j}\,
	\z^{-d-|\a|-2|\eps|+i-|w|+2j}\\
&\qquad\times
\Bigg\{ 2\bigg[
	p \Big( \frac{q_{+}}{\z} \Big)^{p-1} 
\frac{ \Phi_{+}^j }{\z}
	(\z q_{-})^{r}
+
	 r\Big( \frac{q_{+}}{\z} \Big)^{p} 
	 (\z q_{-})^{r-1} ( \z \Phi_{-}^{j} ) \bigg]
	\Big(\frac{1}{\z}\Psi_{+}\Big)^{v} 
	(\z \Psi_{-})^{k-v} 
	\\
& \qquad \quad + 
	\Big( \frac{q_{+}}{\z} \Big)^{p} 
	(\z q_{-})^{r} \bigg[
	v_j s_j \z^{-1}	
	\Big(\frac{1}{\z}\Psi_{+}\Big)^{v-e_j} 
	(\z \Psi_{-})^{k-v} 
	- ( k_j - v_j ) s_j \z 
	\Big(\frac{1}{\z}\Psi_{+}\Big)^{v} 
	(\z \Psi_{-})^{k-v-e_j} 
	\bigg]
	\\
& \qquad \quad +
	\Big( \frac{q_{+}}{\z} \Big)^{p} 
	(\z q_{-})^{r}
	\Big(\frac{1}{\z}\Psi_{+}\Big)^{v} 
	(\z \Psi_{-})^{k-v} 
	\Big( -\frac{1}{2 \z } \Phi_{+}^{j} - \frac{\z}{2} \Phi_{-}^{j} \Big)
\Bigg\} \e\b.
\end{align*}
Proceeding in a similar way as in the proof of \eqref{EST1}, 
this time using also the fact that $|s_j| \le 1$, $j=1,\ldots,d$, and the estimates
\begin{align*}
\Big| \frac{\Phi_{+}^{j}}{\z} \Big| \big( \e\b \big)^{1 \slash 4} \lesssim \z^{-1\slash 2}, \quad
	|\z \Phi^j_{-}|  \big( \e\b \big)^{1 \slash 4} \lesssim \z^{1 \slash 2} \le 
\z^{-1 \slash 2}, \qquad  j=1,\ldots,d,
\end{align*}
which follow from (c) and (d) of Lemma \ref{lem:comp}, we arrive at the desired bound for $H_y$.
\end{proof}

\begin{lem}\label{lem:genStTo}
Let $a>1$, $b>0$ and $M \in \mathbb{R}$ be fixed. Then
\begin{align*}
	\int_{0}^{1} \big( \loo(\z) \big)^{M} (1-\z^2)^{b-1} \, \z^{-a-M} \exp(-T \z^{-1} ) \,
	d\z \lesssim T^{-a+1}, \qquad T>0.
\end{align*}
\end{lem}

\begin{proof}
We split the region of integration onto $(0,1\slash 2)$ and $(1\slash 2,1)$, denoting the resulting
integrals by $J_{0}$ and $J_{1}$, respectively. We first analyze $J_{0}$. For $\zeta \in (0,1\slash 2)$
we have $1-\z^2 \simeq 1$ and $\lo(\z) \simeq \z$. Using this and changing the variable 
$T \z^{-1} \mapsto u$ gives
\begin{align*}
J_0 \simeq \int_0^{1 \slash 2} \z^{-a}\exp(-T\zeta^{-1}) \, d\z
	= T^{-a+1} \int_{2T}^{\infty} u^{a-2} \exp( -u ) \, du
	< T^{-a+1} \int_{0}^{\infty} u^{a-2} \exp( -u ) \, du.
\end{align*}
Since the last integral is finite, we get the required bound for $J_0$.

We next focus on $J_1$. Since $\z \simeq 1$ for $\z \in (1 \slash 2 , 1)$ 
and $\sup_{u\ge 0}u^{a-1}e^{-u}<\infty$, we see that
\begin{align*}
\z^{-a-M}\exp(-T\zeta^{-1})
\lesssim
T^{-a+1} ( T \z^{-1} )^{a-1} \exp ( -T \zeta^{-1})
\lesssim
 T^{-a+1}, \qquad \z \in (1\slash 2,1), \quad T>0.
\end{align*}
This implies the desired estimate for $J_1$ because 
$\int_{1 \slash 2}^{1} \big( \lo(\z) \big)^{M} (1-\z^2)^{b-1} \, d\z < \infty$.
\end{proof}

The next lemma will be applied in Section \ref{sec:ker} with $p=1$, $p=2$ and $p=\infty$.
Other values of $p$ are of interest in connection with operators not considered in this paper, for
instance more general forms of Littlewood-Paley-Stein type square functions.

\begin{lem}\label{lem:genLAI}
Let $d \ge 1$, $\a \in (-1 , \infty)^{d}$, $1 \le p \le \infty$, 
$W \in \mathbb{R}$ and $C>0$. Assume that $\eps \in \{ 0,1 \}^{d}$ and $\vt, \vr \in \{ 0,1,2 \}^{d}$
are such that $\vt \le 2 \eps$ and $\vr \le 2 \eps$. Given $u \ge 0$, we consider the function 
$p_u \colon \R \times \R \times (0,1) \to \mathbb{R}$ defined by
\begin{align*}
& p_{u}(x,y,\z)\\
& \quad=
	(1-\z^2)^{d+|\a|+2|\eps|} \, 
	\z^{-d-|\a|-2|\eps|+|\vt|\slash 2+|\vr|\slash 2-W\slash p-u\slash 2}
	\,x^{2\eps-\vt}y^{2\eps-\vr}
\int \big(\ee\b\big)^{C} \, \1,
\end{align*}
where $W \slash p = 0$ for $p=\infty$. Then $p_u$ satisfies the integral estimate
\begin{align*}
\big\|p_{u}\big(x,y,\z(t)\big)\big\|_{L^{p}(t^{W-1}dt)}
\lesssim
\frac{1}{|x-y|^{u}} \;
\frac{1}{\mu_{\alpha}(B(x,|y-x|))} 
\end{align*}
uniformly in $x,y \in \R$, $x\neq y$, and here $t$ and $\z$ are related as in $\eqref{tzeta}$.
\end{lem}

\begin{proof}
We will show the estimate when $p<\infty$. The case $p=\infty$ can be treated in a similar way,
with the aid of Lemma \ref{lem:comp} (b) instead of Lemma \ref{lem:genStTo} in the reasoning below.

Changing the variable according to \eqref{tzeta} and then using sequently Minkowski's integral inequality,
Lemma \ref{lem:genStTo} (specified to $M=W-1$, $b=p(d+|\alpha|+2|\eps|)$, 
$a=p(d+|\alpha|+2|\eps|-|\vt| \slash 2-|\vr| \slash 2+u \slash 2 ) + 1$, 
$T=\frac{C p q_{+}}{4}$) and the inequality $|x-y|^2 \le q_{+}$, we obtain
\begin{align*}
& \|p_{u}(x,y,\z(t))\|_{L^{p}(t^{W-1}dt)}\\
& \quad =
	x^{2\eps-\vt} y^{2\eps-\vr} 
	\bigg(\int_{0}^{1} \big( \lo(\z) \slash 2 \big)^{W-1} 
	(1-\z^2)^{p(d+|\alpha|+2|\eps|)-1}
	\z^{-p(d+|\alpha|+2|\eps|-|\vt| \slash 2-|\vr| \slash 2+ W\slash p + 
	u \slash 2 )}\\
& \qquad \times
	\bigg( \int \big( \e\b \big)^{C} \, \1 \bigg)^{p} \, d\z \bigg)^{1\slash p}\\
& \quad \lesssim
	x^{2\eps-\vt} y^{2\eps-\vr} \int \bigg( \int_{0}^{1} 
	\big( \lo(\z) \big)^{W-1} (1-\z^2)^{p(d+|\alpha|+2|\eps|)-1}
	\z^{-p(d+|\alpha|+2|\eps|-|\vt| \slash 2-|\vr| \slash 2 + 
	u \slash 2 ) - W }\\
& \qquad \times
	\big(\e\b\big)^{Cp} \, d\z \bigg)^{1\slash p} \, \1 \\
& \quad \lesssim
	x^{2\eps-\vt} y^{2\eps-\vr} 
	\int (q_{+})^{-d-|\a|-2|\eps|+|\vt| \slash 2+|\vr| \slash 2-u \slash 2}
	\, \1 \\
& \quad \le
	\frac{1}{|x-y|^{u}}
	(x+y)^{2(2\eps-\vt \slash 2 - \vr \slash 2)}
	\int (q_{+})^{-d-|\a|-|2\eps - \vt \slash 2 - \vr \slash 2|}
	\, \1.
\end{align*}
Now an application of Lemma \ref{lem:bridge} (taken with $\xi=2\eps-\vt \slash 2-\vr \slash 2$
and $\kappa=\mathbf{1} - \eps + \vt \slash 2+\vr \slash 2$) leads directly to the desired bound.
\end{proof}

We end this section with 
two lemmas that will come into play when proving the smoothness estimates \eqref{sm1} and \eqref{sm2}
(see Section \ref{sec:ker})
in cases when $\mathbb{B}\neq \mathbb{C}$. They will enable us to reduce the difference conditions
to certain gradient estimates, which are easier to verify.
\begin{lem}[{\cite[Lemma 4.5]{Sz}}, {\cite[Lemma 4.3]{Sz1}}] \label{lem:theta}
Let $x,y,z\in\R$ and $s \in [-1,1]^d$. Then 
$$
\frac{1}{4} q_{\pm}(x,y,s) \le q_{\pm}(z,y,s) \le 4 q_{\pm}(x,y,s), 
$$
provided that $|x-y|>2|x-z|$. Similarly, if $|x-y|>2|y-z|$ then
$$
\frac{1}{4} q_{\pm}(x,y,s) \le q_{\pm}(x,z,s) \le 4 q_{\pm}(x,y,s).
$$
\end{lem}

\begin{lem}[{\cite[Lemma 4.5]{Sz1}}] \label{lem:double}
We have  
\begin{align*}
\frac{1}{|z-y|\mu_{\alpha}(B(z,|z-y|))}
\simeq
\frac{1}{|x-y|\mu_{\alpha}(B(x,|x-y|))}
\end{align*}
on the set $\{(x,y,z) \in \R \times \R \times \R : |x-y|>2|x-z|\}$.
\end{lem}

%%%%%%%%%%%%%%%%%%%%%%%%%%%%%%%%%%%%%%%%%%%%%%%%%%%%%%%%%%%%%%%%%%%%%%%%%%%%%%%%%%%%
\section{Kernel estimates} \label{sec:ker}
%%%%%%%%%%%%%%%%%%%%%%%%%%%%%%%%%%%%%%%%%%%%%%%%%%%%%%%%%%%%%%%%%%%%%%%%%%%%%%%%%%%%%

Let $\mathbb{B}$ be a Banach space and let $K(x,y)$ be a kernel defined on 
$\R\times\R\backslash \{(x,y):x=y\}$ and taking values in $\mathbb{B}$.
We say that $K(x,y)$ is a \emph{standard kernel} in the sense of the space of homogeneous type
$(\R, d\mu_{\alpha},|\cdot|)$ if it satisfies so-called \emph{standard estimates}, i{.}e{.}\
the growth estimate
\begin{equation} \label{gr}
\|K(x,y)\|_{\mathbb{B}} \lesssim \frac{1}{\mu_{\alpha}(B(x,|x-y|))}
\end{equation}
and the smoothness estimates
\begin{align}
\| K(x,y)-K(x',y)\|_{\mathbb{B}} & \lesssim \frac{|x-x'|}{|x-y|}\, \frac{1}{\mu_{\alpha}(B(x,|x-y|))},
\qquad |x-y|>2|x-x'|, \label{sm1}\\
\| K(x,y)-K(x,y')\|_{\mathbb{B}} & \lesssim \frac{|y-y'|}{|x-y|}\, \frac{1}{\mu_{\alpha}(B(x,|x-y|))},
\qquad |x-y|>2|y-y'| \label{sm2}.
\end{align}
When $K(x,y)$ is scalar-valued, i.e.\ $\mathbb{B}=\mathbb{C}$, the difference bounds \eqref{sm1}
and \eqref{sm2} are implied by the more convenient gradient estimate
\begin{equation} \label{grad}
|\nabla_{\! x,y} K(x,y)| \lesssim \frac{1}{|x-y|\mu_{\alpha}(B(x,|x-y|))}.
\end{equation}
Notice that in these formulas, the ball $B(x,|y-x|)$ can be replaced by $B(y,|x-y|)$, in view of
the doubling property of $\mu_{\alpha}$.

We will show that the following kernels, valued in suitably chosen Banach spaces $\mathbb{B}$, satisfy the
standard estimates.
\begin{itemize}
\item[(1)] The kernel associated to the Laguerre heat semigroup maximal operator,
$$
\mathcal{G}^{\alpha}(x,y) = \big\{G_t^{\alpha}(x,y)\big\}_{t>0}, \qquad \mathbb{B}=L^{\infty}(dt).
$$
\item[(2)] The kernels associated with Riesz-Laguerre transforms,
$$
R_n^{\alpha}(x,y) = \frac{1}{\Gamma(|n|\slash 2)} \int_0^{\infty} \delta_x^n G_t^{\alpha}(x,y)
	t^{|n|\slash 2 -1}\, dt, \qquad \mathbb{B}=\mathbb{C},
$$
where $n \in \N^d$ is such that $|n| > 0$. 
\item[(3)] The kernels associated with mixed square functions,
$$
\mathcal{H}^{\alpha}_{n,m}(x,y) = \big\{ \partial_t^m \delta_x^n G_t^{\alpha}(x,y) \big\}_{t>0}, \qquad
	\mathbb{B} = L^2(t^{|n|+2m-1}dt),
$$
where $n \in \N^d$ and $m \in \N$ are such that $|n|+m>0$.
\item[(4)] The kernels associated to Laplace transform type multipliers,
$$
K^{\alpha}_{\psi}(x,y) = - \int_0^{\infty} \psi(t) \partial_t G_t^{\alpha}(x,y)\, dt, \qquad
	\mathbb{B}=\mathbb{C},
$$
where $\psi \in L^{\infty}(dt)$.
\item[(5)] The kernels associated to Laplace-Stieltjes transform type multipliers,
$$
K^{\alpha}_{\nu}(x,y) =  \int_0^{\infty} G_t^{\alpha}(x,y)\, d\nu(t), \qquad
	\mathbb{B}=\mathbb{C},
$$
where $\nu$ is a signed or complex Borel measure on $(0,\infty)$ 
with total variation $|\nu|$
satisfying 
\begin{equation}\label{assum}
\int_{0}^{\infty} e^{-t(2d + 2|\a| )} \, d|\nu|(t) < \infty.
\end{equation}
\end{itemize}

The result below extends to all $\alpha \in (-1,\infty)^d$ analogous estimates obtained in
\cite{NS1,Sz,Sz2} for $\alpha \in [-1\slash 2,\infty)^d$ (to be precise, 
$\mathcal{H}^{\alpha}_{n,m}(x,y)$ was estimated in \cite{Sz} only in the special cases when
either $|n|=1$ and $m=0$ or $|n|=0$ and $m=1$; here we obtain a more general result).
\begin{thm} \label{thm:std}
Let $\alpha \in (-1,\infty)^d$. Then the kernels (1)-(5) satisfy the standard estimates \eqref{gr},
\eqref{sm1} and \eqref{sm2} with $\mathbb{B}$ as indicated above.
\end{thm}

The Laguerre-Poisson semigroup is given by the Laguerre-Poisson kernel $P_t^{\alpha}(x,y)$, which
is linked to the heat kernel by the subordination formula,
$$
P_t^{\alpha}(x,y) = \int_0^{\infty} G^{\alpha}_{t^2\slash (4u)}(x,y) \; \frac{e^{-u}du}{\sqrt{\pi u}}.
$$
Our technique, presented in a moment in the proof of Theorem \ref{thm:std}, works perfectly
also for kernels emerging from $P_t^{\alpha}(x,y)$, and only slightly more effort is needed
(see for instance \cite[Section 4.3]{Sz}). In particular, it can be proved that the kernels in
(1) and (4) with $G_t^{\alpha}(x,y)$ replaced by $P_t^{\alpha}(x,y)$ satisfy the standard estimates.
The same is true about the kernel in (3) if $\mathbb{B}$ corresponding to 
$\{ \partial_t^m \delta_x^n P_t^{\alpha}(x,y)\}$ is chosen as $L^2(t^{2|n|+2m-1}dt)$, and about 
the kernel in (5) if we replace $2d + 2|\a|$ in \eqref{assum} by $\sqrt{ 2d + 2|\a| }$.
We leave details to interested readers.

The remaining part of this section is devoted to the proof of Theorem \ref{thm:std}.
In the proof we tacitly assume that passing
with the differentiation in $t$, $x_j$ or $y_j$ under integrals against $\1$, $dt$ or $d\nu(t)$
is legitimate. This is indeed always the case, as can be easily justified with the aid of the estimates
obtained in Lemma \ref{lem:EST} and in the proof of Theorem \ref{thm:std}; 
see \cite[Section 5]{NS1} and \cite[Section 4]{Sz}, where the details are given in the contexts 
of Riesz transforms and $g$-functions, respectively.

%%%%%%%%%%%%%%%%%%%%%%%%%%%%%%%%%%%%%%%%%%%%%%%%%%%%%%%%%%%%%%%%%%%%%%%%%%%%%%%%%%%%%%%%%%%%%%
\begin{proof}[Proof of Theorem \ref{thm:std}; the case of $\mathcal{G}^{\alpha}(x,y)$]
In view of \eqref{G}, the growth condition for $\mathcal{G}^{\alpha}(x,y)$ is a direct consequence 
of Lemma \ref{lem:genLAI} (specified to $u=0$, $p=\infty$, $W=C=1$, $\vt = \vr = 0$).

To prove the smoothness estimates it is enough, by symmetry reasons, to show \eqref{sm1}.
By the Mean Value Theorem
$$
\big| G_t^{\a}(x,y) - G_t^{\a}(x',y) \big| 
\le 
|x-x'| \big| \nabla_{\!x} G_t^{\a}(x,y) \big|_{x=\t} \big|,
$$
where $\t$ is a convex combination of $x$, $x'$ that depends also on $t$. Thus it suffices to verify that
$$
\Big\| \big | \nabla_{\!x} G_t^{\a}(x,y) \big|_{x=\t} \big| \Big\|_{L^{\infty} (dt) }
\lesssim
\frac{1}{ |x-y| \mu_{\alpha}(B(x,|x-y|))},\qquad |x-y|>2|x-x'|.
$$
Observe that $\t \le {x \vee x'}$, $|x-\t| \le |x-x'|$ and $|x-{x \vee x'}| \le |x-x'|$. 
Applying \eqref{EST2} of Lemma \ref{lem:EST} (taken with $n=(0,\ldots,0)$ 
and $m=0$) and Lemma \ref{lem:theta} (first with $z=\t$ and then with $z={x \vee x'}$) we obtain
\begin{align*}
& \big| \nabla_{\!x} G_t^{\a}(x,y) \big|_{x=\t} \big|\\
& \quad \lesssim
	\sum_{\eps \in \{0,1\}^{d}} (1-\zeta^2)^{d+|\alpha|+2|\eps|} y^{2\eps} 
	\sum_{\eta \in \{0,1,2\}^d} ({x \vee x'})^{2\eps-\eta\eps} 
	\zeta^{-d-|\a|-2|\eps|+|\eta\eps|\slash 2 -1 \slash 2}\\
& \quad \qquad \times
	\int \big( \e\eee \big)^{1\slash 64} \, \1 \\
& \qquad +
	\sum_{\eps \in \{0,1\}^{d}} (1-\zeta^2)^{d+|\alpha|+2|\eps|} 
	\sum_{j=1}^{d} \chi_{ \{ \eps_j = 1 \} }y^{2\eps-e_j} 
	\sum_{\eta \in \{0,1,2\}^d} ({x \vee x'})^{2\eps-\eta\eps} 
	\zeta^{-d-|\a|-2|\eps|+|\eta\eps|\slash 2}\\
& \quad \qquad \times
	\int \big( \e\eee \big)^{1\slash 64} \, \1, 
\end{align*}
provided that $|x-y|>2|x-x'|$. Now the conclusion follows with the aid of Lemma \ref{lem:genLAI} (applied
with $u=1$, $p=\infty$, $W=1$, $C=1\slash 64$, $\vt = \eta\eps$ and either $\vr = 0$ or $\vr = e_j$)
and Lemma \ref{lem:double} specified to $z={x \vee x'}$.
\end{proof}

%%%%%%%%%%%%%%%%%%%%%%%%%%%%%%%%%%%%%%%%%%%%%%%%%%%%%%%%%%%%%%%%%%%%%%%%%%%%%%%%%%%%%%%%%%%%%%%%%%%%%%%%%%%
\begin{proof}[Proof of Theorem \ref{thm:std}; the case of ${R}_n^{\alpha}(x,y)$]
The growth estimate \eqref{gr} follows immediately from \eqref{EST1} of Lemma \ref{lem:EST} taken 
with $m=0$ and Lemma \ref{lem:genLAI} (applied with $u=0$, $p=1$, 
$W=|n| \slash 2$, $C=1\slash 2$, $\vt = \eta\eps$ and $\vr = 0$). 

To prove the gradient condition \eqref{grad}, it suffices to check that
$$
\Big\| \big| \nabla_{\!x,y} \delta_{x}^{n} G_t^{\alpha}(x,y) \big| \Big\|_{L^1 (t^{|n|\slash 2 -1} dt)}
\lesssim
\frac{1}{|x-y|\mu_{\alpha}(B(x,|x-y|))}, \qquad x\ne y.
$$
This estimate, however, follows readily by combining \eqref{EST2} of Lemma \ref{lem:EST}
(specified to $m=0$) with Lemma \ref{lem:genLAI} (taken with $u=1$, $p=1$, $W=|n| \slash 2$, 
$C=1\slash 4$, $\vt = \eta\eps$ and either $\vr = 0$ or $\vr = e_j$). 
\end{proof}

%%%%%%%%%%%%%%%%%%%%%%%%%%%%%%%%%%%%%%%%%%%%%%%%%%%%%%%%%%%%%%%%%%%%%%%%%%%%%%%%%%%%%%%%%%%%%%%%%%%%%%%%%%%
\begin{proof}[Proof of Theorem \ref{thm:std}; the case of $\mathcal{H}_{n,m}^{\alpha}(x,y)$]
The growth condition follows by using \eqref{EST1} of Lemma \ref{lem:EST}
and then Lemma \ref{lem:genLAI} (specified to $u=0$, $p=2$, $W=|n| + 2m$, $C=1\slash 2$, 
$\vt = \eta\eps$, $\vr = 0$).

Next, we verify the smoothness bound \eqref{sm1}. 
Proving the other smoothness estimate relies on essentially the same arguments and is left to the reader.

By the Mean Value Theorem it suffices to show that
$$
\Big\| \big| \nabla_{\!x} \partial_{t}^{m} \delta_{x}^{n} G_t^{\alpha}(x,y)\big|_{x=\t} 
\big| \Big\|_{L^2(t^{|n| +2m - 1} dt)}
\lesssim
\frac{1}{|x-y|\mu_{\alpha}(B(x,|x-y|))}, \qquad |x - y|>2|x-x'|,
$$
where $\t$ is a convex combination of $x$ and $x'$ that depends also on $t$. 
Using \eqref{EST2} of Lemma \ref{lem:EST}, the inequalities 
$\t \le {x \vee x'}$, $|x-\t| \le |x-x'|$, $|x-{x \vee x'}| \le |x-x'|$ and 
Lemma \ref{lem:theta} twice (with $z=\t$ and $z={x \vee x'}$) we get
\begin{align*}
& \big| \nabla_{\!x} \partial_{t}^{m} \delta_{x}^{n} G_t^{\alpha}(x,y)\big|_{x=\t} \big| \\
& \quad \lesssim
\sum_{\eps \in \{0,1\}^{d}} (1-\zeta^2)^{d+|\alpha|+2|\eps|} y^{2\eps} 
	\sum_{\eta \in \{0,1,2\}^d} ({x \vee x'})^{2\eps-\eta\eps} 
	\zeta^{-d-|\a|-2|\eps|-m - |n| \slash 2 +|\eta\eps|\slash 2 -1 \slash 2}\\
& \quad \qquad \times
	\int \big( \e\eee \big)^{1\slash 64} \, \1 \\
& \qquad +
	\sum_{\eps \in \{0,1\}^{d}} (1-\zeta^2)^{d+|\alpha|+2|\eps|} 
	\sum_{j=1}^{d} \chi_{ \{ \eps_j = 1 \} } y^{2\eps-e_j} 
	\sum_{\eta \in \{0,1,2\}^d} ({x \vee x'})^{2\eps-\eta\eps} 
	\zeta^{-d-|\a|-2|\eps| - m - |n| \slash 2 +|\eta\eps|\slash 2}\\
& \quad \qquad \times
	\int \big( \e\eee \big)^{1\slash 64} \, \1, 
\end{align*}
provided that $|x-y|>2|x-x'|$.
This, together with Lemma \ref{lem:genLAI} (specified to $u=1$, $p=2$, $W=|n| + 2m$, $C=1\slash 64$,
$\vt = \eta\eps$ and either $\vr = 0$ or $\vr = e_j$) and Lemma \ref{lem:double} 
(applied with $z={x \vee x'}$), produces the desired bound. 
\end{proof}

%%%%%%%%%%%%%%%%%%%%%%%%%%%%%%%%%%%%%%%%%%%%%%%%%%%%%%%%%%%%%%%%%%%%%%%%%%%%%%%%%%%%%%%%%%%%%%%%%%%%%%%%%%%
\begin{proof}[Proof of Theorem \ref{thm:std}; the case of ${K}_{\psi}^{\alpha}(x,y)$]
The growth condition is a straightforward consequence of \eqref{EST1} of Lemma \ref{lem:EST}
(taken with $n=(0,\ldots,0)$ and $m=1$), the fact that $\psi \in L^{\infty}(dt)$ and Lemma \ref{lem:genLAI} 
(specified to $u=0$, $p=1$, $W=1$, $C=1\slash 2$, $\vt = \eta\eps$, $\vr = 0$).

To prove the gradient condition, in view of the boundedness of $\psi$, it suffices to verify that
$$
\Big\| \big| \nabla_{\!x,y} \partial_{t} G_t^{\a}(x,y)  \big| \Big\|_{L^1(dt)}
\lesssim
\frac{1}{|x-y|\mu_{\alpha}(B(x,|x-y|))}, \qquad x \ne y.
$$
This, however, follows immediately from \eqref{EST2} of 
Lemma \ref{lem:EST} (with $n=(0,\ldots,0)$ and $m=1$) 
and Lemma \ref{lem:genLAI} (applied with $u=1$, $p=1$, 
$W=1$, $C=1\slash 4$, $\vt = \eta\eps$ and either $\vr = 0$ or $\vr = e_j$). 
\end{proof}

%%%%%%%%%%%%%%%%%%%%%%%%%%%%%%%%%%%%%%%%%%%%%%%%%%%%%%%%%%%%%%%%%%%%%%%%%%%%%%%%%%%%%%%%%%%%%%%%%%%%%%%%%%%
\begin{proof}[Proof of Theorem \ref{thm:std}; the case of ${K}_{\nu}^{\alpha}(x,y)$]
In order to show the growth bound it is enough, by the assumption \eqref{assum} concerning $\nu$, 
to check that
$$
e^{t(2d+2|\a|)}G_t^\a(x,y)
\lesssim 
\frac{1}{\mu_{\alpha}(B(x,|x-y|))}, \qquad x \ne y, \quad t>0.
$$
Taking into account \eqref{G}, an application of Lemma \ref{lem:comp} (b) 
(specified to $b=d + |\a| + 2|\eps|$, $c= 1\slash 4$, $A=\z^{-1}$) gives
\begin{align*}
e^{t(2d+2|\a|)}G_t^\a(x,y)
& \lesssim
\sum_{\eps \in \{ 0,1 \}^d} \z^{-d-|\a|-2|\eps|} (xy)^{2 \eps}
	\int \e\b \1 \\
& \lesssim
\sum_{\eps \in \{ 0,1 \}^d} (x+y)^{4 \eps} \int (q_+)^{-d - |\a| - 2|\eps|} \1.
\end{align*}
This, in view of Lemma \ref{lem:bridge} 
(applied with $\xi = 2\eps$, $\kappa = \mathbf{1} - \eps$), leads to the desired conclusion. 

To justify the gradient estimate \eqref{grad}, it suffices to verify that
$$
e^{t(2d+2|\a|)} \big| \nabla_{\!x,y}G_t^\a(x,y) \big|
\lesssim 
\frac{1}{|x-y|\mu_{\alpha}(B(x,|x-y|))}, \qquad x \ne y, \quad t>0.
$$
Proceeding in a similar way as in the case of the growth condition, using this time 
\eqref{EST2} of Lemma \ref{lem:EST} (applied with $n=(0,\ldots,0)$ and $m=0$) and Lemma \ref{lem:comp} (b) 
(specified to $c= 1\slash 16$, $A=\z^{-1}$ and either 
$b=d + |\a| + 2|\eps| - |\eta\eps| \slash 2 + 1 \slash 2$ or 
$b=d + |\a| + 2|\eps| - |\eta\eps| \slash 2$) we see that
\begin{align*}
& e^{t(2d+2|\a|)} \big| \nabla_{\!x,y}G_t^\a(x,y) \big| \\
& \, \lesssim
\sum_{\eps \in \{ 0,1 \}^d} 
	\sum_{\eta \in \{0,1,2\}^d}
	(x+y)^{2(2\eps-\eta\eps \slash 2)} 
	\int (q_+)^{-d - |\a| - |2\eps - \eta\eps \slash 2| - 1 \slash 2} \1 \\
& \quad + \!
	\sum_{\eps \in \{ 0,1 \}^d} \sum_{j=1}^{d} \chi_{\{\eps_j=1\}} \!
	\sum_{\eta \in \{0,1,2\}^d}
	(x+y)^{2(2\eps-\eta\eps \slash 2 -e_j \slash 2)} 
	\int (q_+)^{-d - |\a| - |2\eps - \eta\eps \slash 2 - e_j \slash 2| - 1 \slash 2} \1.
\end{align*}
Finally, in view of Lemma \ref{lem:bridge} (taken with $\xi = 2\eps - \eta\eps \slash 2$, $\kappa =
\mathbf{1} - \eps + \eta\eps \slash 2$ and $\xi = 2\eps - \eta\eps \slash 2 - e_j \slash 2$, $\kappa =
\mathbf{1} - \eps + \eta\eps \slash 2 +  e_j \slash 2$), we arrive at the required bound. 
\end{proof}

The proof of Theorem \ref{thm:std} is complete.

%%%%%%%%%%%%%%%%%%%%%%%%%%%%%%%%%%%%%%%%%%%%%%%%%%%%%%%%%%%%%%%%%%%%%%%%%%%%%%%%%%%%
\section{Conclusions} \label{sec:CZ}
%%%%%%%%%%%%%%%%%%%%%%%%%%%%%%%%%%%%%%%%%%%%%%%%%%%%%%%%%%%%%%%%%%%%%%%%%%%%%%%%%%%%%

Let $\mathbb{B}$ be a Banach space and
suppose that $T$ is a linear operator assigning to each $f\in L^2(d\mu_{\alpha})$
a strongly measurable $\mathbb{B}$-valued function $Tf$ on $\R$. Then $T$ is said to be a (vector-valued)
Calder\'on-Zygmund operator in the sense of the space $(\R,d\mu_{\alpha},|\cdot|)$ associated with
$\mathbb{B}$ if
\begin{itemize}
\item[(A)] $T$ is bounded from $L^2(d\mu_{\alpha})$ to $L^2_{\mathbb{B}}(d\mu_{\alpha})$,
\item[(B)] there exists a standard $\mathbb{B}$-valued kernel $K(x,y)$ such that
$$
Tf(x) = \int_{\R} K(x,y) f(y)\, d\mu_{\alpha}(y), \qquad \textrm{a.e.}\; x \notin \support f,
$$
for every $f \in L_c^{\infty}(\R)$,
where $L_c^{\infty}(\R)$ is the subspace of $L^{\infty}(\R)$ of bounded measurable functions
with compact supports.
\end{itemize}
Here integration of $\mathbb{B}$-valued functions is understood in Bochner's sense, and
$L^2_{\mathbb{B}}(d\mu_{\alpha})$ is the Bochner-Lebesgue space of all $\mathbb{B}$-valued 
$d\mu_{\alpha}$-square integrable functions on $\R$.

It is well known that a large part of the classical theory of Calder\'on-Zygmund operators remains valid,
with appropriate adjustments, when the underlying space is of homogeneous type and the associated kernels
are vector-valued, see for instance \cite[p.\,649]{NS1} and references given there. 
In particular, if $T$ is a Calder\'on-Zygmund
operator in the sense of $(\R,d\mu_{\alpha},|\cdot|)$ associated with a Banach space $\mathbb{B}$,
then its mapping properties in weighted $L^p$ spaces follow from the general theory; 
see \cite[Section 2]{BCN}.

Let 
$$
T_t^{\alpha}f(x) = \int_{\R} G_t^{\alpha}(x,y)f(y)\, d\mu_{\alpha}(y), \qquad t>0, \quad x\in \R.
$$
For $\alpha \in (-1,\infty)^d$ consider the following operators defined initially in $L^2(d\mu_{\alpha})$.
\begin{itemize}
\item[(1)] The Laguerre heat semigroup maximal operator
$$
T_{*}^{\alpha}f = \big\| T_t^{\alpha}f\big\|_{L^{\infty}(dt)}.
$$
\item[(2)] Riesz-Laguerre transforms of order $|n|>0$
$$
R_n^{\alpha}f = \sum_{k \in \N^d} \big( 4|k| + 2|\alpha| + 2d\big)^{-|n|\slash 2}
	\langle f , \ell_k^{\alpha}\rangle_{d\mu_{\alpha}}\, \delta^n \ell_k^{\alpha},
$$
where $n \in \N^d$
and $\langle f , \ell_k^{\alpha}\rangle_{d\mu_{\alpha}}$ are the Fourier-Laguerre coefficients of $f$.
\item[(3)] Littlewood-Paley-Stein type mixed square functions 
$$
g^{\alpha}_{n,m}(f) = \big\| \partial_t^m \delta^n T_t^{\alpha}f \big\|_{L^2(t^{|n|+2m-1}dt)},
$$
where $n \in \N^d$, $m \in \N$, $|n|+m>0$.
\item[(4)] Multipliers of Laplace transform type
$$
M^{\alpha}_{\mathfrak{m}} f = \sum_{k \in \N^d} \mathfrak{m}(4|k|+2|\alpha|+2d) 
		\langle f , \ell_k^{\alpha}\rangle_{d\mu_{\alpha}}\,  \ell_k^{\alpha},
$$
where $\mathfrak{m}(z) = z\int_0^{\infty} e^{-tz} \psi(t)\, dt$ with $\psi \in L^{\infty}(dt)$.
\item[(5)] Multipliers of Laplace-Stieltjes transform type
$$
M^{\alpha}_{\mathfrak{m}} f = \sum_{k \in \N^d} \mathfrak{m}(4|k|+2|\alpha|+2d) 
		\langle f , \ell_k^{\alpha}\rangle_{d\mu_{\alpha}}\,  \ell_k^{\alpha},
$$
where $\mathfrak{m}(z) = \int_0^{\infty} e^{-tz} \, d\nu (t)$ with 
$\nu$ being a signed or complex Borel measure on $(0,\infty)$, 
with its total variation $|\nu|$ satisfying 
$$
\int_{0}^{\infty} e^{-t(2d + 2|\a| )} \, d|\nu|(t) < \infty.
$$
\end{itemize}
We remark that the formulas defining $T^{\alpha}_{*}f$ and $g_{n,m}^{\alpha}(f)$ are valid
(the integral defining $T_t^{\alpha}f(x)$ converges and produces a smooth function of 
$(x,t)\in \R\times\mathbb{R}_{+}$) 
for general functions $f$ from weighted $L^p$ spaces and
Muckenhoupt weights; see \cite[p.\,648]{NS1} and \cite[Section 2]{Sz} for the relevant arguments.

As a consequence of Theorem \ref{thm:std} we get the following result.
\begin{thm} \label{thm:main}
Let $\alpha \in (-1,\infty)^d$. The Riesz-Laguerre transforms and the multipliers of Laplace
and Laplace-Stieltjes
transforms type are scalar-valued Calder\'on-Zygmund operators in the sense of 
the space $(\R,d\mu_{\alpha},|\cdot|)$.
Furthermore, the Laguerre heat semigroup maximal operator and the mixed square functions can be viewed as
vector-valued Calder\'on-Zygmund operators in the sense of $(\R,d\mu_{\alpha},|\cdot|)$ associated
with Banach spaces $\mathbb{B}=C_0$ and $\mathbb{B}=L^2(t^{|n|+2m-1}dt)$, respectively, where
$C_0$ is a separable subspace of $L^{\infty}(dt)$ consisting of all continuous functions $f$ on
$\mathbb{R}_{+}$ which have finite limits as $t\to 0^{+}$ and vanish as $t\to \infty$.
\end{thm}

\begin{proof}
The standard estimates are provided in all the cases by Theorem \ref{thm:std}.
Thus it suffices to verify $L^2$-boundedness and kernel associations (conditions (A) and (B) above).
This, however, was essentially done already 
in \cite{NS1,Sz,Sz2} since the arguments given there are actually
valid for all $\alpha \in (-1,\infty)^d$ provided that the same is true about the standard estimates.
To be precise, 
an exception here are the mixed square functions because the arguments from \cite{Sz} cover only some
special cases. Proving the desired properties in the general case requires in addition the decomposition
from \cite[Proposition 3.5]{NS1}; see \cite[Section 4.2]{BCN} where the relevant arguments were given in
the setting of continuous Bessel expansions.
\end{proof}

Denote by $A_p^{\alpha}$, $1\le p < \infty$, the Muckenhoupt classes of weights related to
$(\R,d\mu_{\alpha},|\cdot|)$ (for the definition, see for instance \cite[p.\,645]{NS1}).
\begin{cor} \label{cor:main}
Let $\alpha \in (-1,\infty)^d$. The Riesz-Laguerre transforms and the multipliers of Laplace
and Laplace-Stieltjes
types extend to bounded linear operators on $L^p(wd\mu_{\alpha})$, $w \in A_p^{\alpha}$,
$1<p<\infty$, and from $L^1(wd\mu_{\alpha})$ to weak $L^1(wd\mu_{\alpha})$, $w \in A_1^{\alpha}$.
Furthermore, the Laguerre heat semigroup maximal operator and the mixed square functions, 
viewed as scalar-valued
sublinear operators, are bounded on $L^p(wd\mu_{\alpha})$, $w \in A_p^{\alpha}$,
$1<p<\infty$, and from $L^1(wd\mu_{\alpha})$ to weak $L^1(wd\mu_{\alpha})$, $w \in A_1^{\alpha}$.
\end{cor}

\begin{proof}
The part concerning $R_n^{\alpha}$ and $M_{\mathfrak{m}}^{\alpha}$ is a direct consequence of Theorem
\ref{thm:main} and the general theory. The remaining part requires some additional, but standard arguments,
see the proof of \cite[Theorem 2.1]{NS1} and \cite[Corollary 2.5]{Sz}. We leave details to interested
readers.
\end{proof}

Finally, we remark that results parallel to Theorem \ref{thm:main} and Corollary \ref{cor:main}
are in force for the Poisson semigroup based analogues of $T_{*}^{\alpha}$, $g_{n,m}^{\alpha}$ and
$M_{\mathfrak{m}}^{\alpha}$, see the comment following the statement of Theorem \ref{thm:std}.
This follows by quite obvious adjustments of the arguments for the heat semigroup based objects
and hence the details are omitted.


\begin{thebibliography}{99}

\bibitem{BCN}
J{.}J{.} Betancor, A{.}J{.} Castro, A{.} Nowak,
\emph{Calder\'on-Zygmund operators in the Bessel setting},
preprint 2010. \texttt{arXiv:1012.5638v1}

\bibitem{CW}      
R{.} Coifman and G{.} Weiss,
\emph{Analyse harmonique non-commutative sur certains espaces homog\`enes},
Lecture Notes in Mathematics {242}, Springer, Berlin-New York, 1971.

\bibitem{Jo}
W{.}P{.} Johnson,
\emph{The curious history of Fa\`a di Bruno's formula},
Amer{.} Math{.} Monthly 109 (2002), 217-–234.

\bibitem{NoSj}
A{.} Nowak and P{.} Sj\"ogren,
\emph{Calder\'on-Zygmund operators related to Jacobi expansions},
preprint 2010. \texttt{arXiv:1011.3615v1}

\bibitem{NS1} 
A{.} Nowak and K{.} Stempak, 
\emph{Riesz transforms for multi-dimensional Laguerre function expansions}, 
Adv{.} Math{.} 215 (2007), 642--678.

\bibitem {NS3} 
A{.} Nowak and K{.} Stempak, 
\emph{Riesz transforms  for  the Dunkl harmonic oscillator}, 
Math{.} Z{.} 262 (2009), 539--556. 

\bibitem{NS2} 
A{.} Nowak and K{.} Stempak,
\emph{Imaginary powers of the Dunkl harmonic oscillator}, 
Symmetry, Integrability and Geometry: Methods and Applications; SIGMA 5 (2009), 016, 12 pages.

\bibitem{Sa}
E{.} Sasso,
\emph{Functional calculus for the Laguerre operator},
Math{.} Z{.} 249 (2005), 683--711.

\bibitem{Sz}
T{.} Szarek,
\emph{Littlewood-Paley-Stein type square functions based on Laguerre semigroups}, 
Acta Math{.} Hungar{.}, in press. Online First version DOI: 10.1007/s10474-010-0016-8.

\bibitem{Sz1} 
T{.} Szarek,
\emph{On Lusin's area integrals and $g$-functions in certain Dunkl and Laguerre settings}, 
preprint 2010. \texttt{arXiv:1011.0898v1}

\bibitem{Sz2}
T{.} Szarek,
\emph{Multipliers of Laplace transform type in certain Dunkl and Laguerre settings},
preprint 2010. \texttt{arXiv:1101.4139v1}

\bibitem{Wat}     
G{.}N{.} Watson,
\emph{A treatise on the theory of Bessel functions},
Cambridge University Press, Cambridge, 1966.

\end{thebibliography}
\end{document}